\documentclass[11pt,amssymb]{amsart}
\usepackage{amssymb}
\usepackage{amsmath}
\usepackage[mathscr]{eucal}

\newcommand{\Z}{{\mathbb Z}}

\newcommand{\Q}{{\mathbb Q}}

\newcommand{\R}{{\mathbb R}}

\newcommand{\Ga}{\mathrm{Gal}}
\newtheorem{thm}{Theorem}[section]
\newtheorem{lemma}[thm]{Lemma}
\newtheorem{prop}[thm]{Proposition}
\newtheorem{cor}[thm]{Corollary}

\newcommand{\Hom}{\mathrm{Hom}}

\newcommand{\Fpm}{\mathrm{(F}_m'\mathrm{)}}
\newcommand{\tF}{\mathrm{(F)}}
\newcommand{\he}{H_{\mathrm{\acute{e}t}}}
\newcommand{\hhe}{H^{\mathrm{\acute{e}t}}}

\usepackage[all,cmtip]{xy}
.240pk scaled 1200 .240pk
\begin{document}

\title[Finiteness of unramified cohomology]{On some finiteness results for unramified cohomology}

\author[I.~Rapinchuk]{Igor A. Rapinchuk}

\begin{abstract}
We obtain several finiteness results for the unramified cohomology of function fields of algebraic varieties defined over fields of type $\Fpm$, a class that includes algebraically closed fields, finite fields, local fields, and some higher local fields of characteristic 0.
\end{abstract}

\address{Department of Mathematics, Harvard University, Cambridge, MA, 02138 USA}

\email{rapinch@math.harvard.edu}

\maketitle

\section{Introduction}\label{S-1}

Let $K$ be a field equipped with a discrete valuation $v$. It is well-known that for any positive integer $m$ invertible in the residue field $\kappa (v)$ and any $i \geq 1$, there exists a {\it residue map} in Galois cohomology
$$
\partial_v^i \colon H^i (K, \mu_m^{\otimes j}) \to H^{i-1} (\kappa (v), \mu_m^{\otimes (j-1)})
$$
(see the end of this section for all unexplained notations and \cite[\S 3.3]{CT-SB} for a description of several constructions of $\partial_v^i$). A cohomology class $x \in H^i (K, \mu_m^{\otimes j})$ is said to be {\it unramified at $v$} if $x \in \ker \partial_v^i.$ Furthermore, if $V$ is a set of discrete valuations of $K$ such that the maps $\partial_v^i$ exist for all $v \in V$, one defines the {\it degree $i$ unramified cohomology of $K$ with respect to $V$} as
$$
H^i (K, \mu_m^{\otimes j})_V = \bigcap_{v \in V} \ker \partial_v^i.
$$

In this note, we will be primarily concerned with unramified cohomology groups that arise in geometric situations.
Namely, suppose $X$ is a smooth irreducible algebraic variety over a field $F$ with function field $F(X).$ Then each point $x \in X$ of codimension 1 defines a discrete valuation $v_x$ on $F(X)$ that is trivial on $F$. We let
$$
V_0 = \{v_x \mid x \in X^{(1)} \}
$$
denote the set of all such {\it geometric places} of $F(X)$ and define
$$
H^i_{\mathrm{ur}} (F(X), \mu_m^{\otimes j}) = H^i (F(X), \mu_m^{\otimes j})_{V_0}
$$
for any positive integer $m$ invertible in $F$.\footnotemark \footnotetext{Another definition of unramified cohomology that is frequently encountered is
$$
H^i_{\mathrm{nr}} (F(X), \mu_m^{\otimes j}) = H^i (F(X), \mu_m^{\otimes j})_{V_1},
$$
where $V_1$ is the set of all discrete valuations $v$ of $F(X)$ such that $F$ is contained in the valuation ring $\mathcal{O}_v$.
Clearly, there is an inclusion $H^i_{\mathrm{nr}} (F(X), \mu_m^{\otimes j}) \subset H^i_{\mathrm{ur}} (F(X), \mu_m^{\otimes j})$, which is in fact an equality if $X$ is proper (see \cite[Theorem 4.1.1]{CT-SB}). Note that by construction, $H^i_{\mathrm{nr}} (F(X), \mu_m^{\otimes j})$ is a birational invariant of $X$.}

Unramified cohomology in degree 2 (i.e. the {\it unramified Brauer group}) was initially introduced by Saltman \cite{Sa} in connection with his work on Noether's problem. Later, Colliot-Th\'el\`ene and Ojanguren \cite{CTO} gave a general definition of unramified cohomology, which they then employed to produce new examples of algebraic varieties that are unirational but nonrational. Since then, unramified cohomology has become an important tool in a variety of problems involving algebraic groups, division algebras, quadratic forms, and rationality, among others (see, e.g., \cite{CT-SB}, \cite{GaMeSe}, \cite{M}, \cite{Salt}).

Over the last several years, unramified cohomology has been applied to the study of division algebras having the same maximal subfields, and, more generally, linear algebraic groups having the same isomorphism classes of maximal tori. We refer the reader to \cite{CRR3} for a detailed overview of these problems, as well as a description of recent results and conjectures. A question that comes up in this context is

\vskip2mm

$(\dag)$ \parbox[t]{15.5cm}{{\it Given a ``nice" field $K$, does there exist some naturally defined set $V$ of discrete valuations of $K$ such that the groups $H^i (K, \mu_m^{\otimes j})_V$ are finite for all $i \geq 1$?}}

\vskip2mm

In this note, we address $(\dag)$ for function fields of algebraic varieties over base fields satisfying the following condition. Given an integer $m \geq 1$, we will say that a field $K$ is {\it of type $\Fpm$} if
\vskip2mm

$\Fpm$ \ \ \ \parbox[t]{15.5cm}{{\it For every finite separable extension $L/K$, the quotient $L^{\times}/{L^{\times}}^m$ is finite.}}

\vskip2mm
\noindent Some common examples of fields of type $\Fpm$ include finite fields and $p$-adic fields, as well as higher local fields like $\Q_p ((t_1)) \cdots ((t_r))$ (see \S \ref{S-2}).
Our main result is

\begin{thm}\label{T-MainThm}
Let $K$ be a field and $m \geq 1$ an integer prime to $\mathrm{char}~K.$ Assume that $K$ is of type $\Fpm.$
\vskip2mm

\noindent $\mathrm{(a)}$ \parbox[t]{16cm}{Suppose $C$ is a smooth, geometrically integral curve over $K$. Then the
unramified cohomology groups $H^i_{\mathrm{ur}} (K(C), \mu_m^{\otimes j})$ are finite for all $i \geq 1.$}

\vskip3mm

\noindent $\mathrm{(b)}$ \parbox[t]{16cm}{Let $X$ be a smooth, geometrically integral algebraic variety of dimension $\geq 2$ over $K$. Then the unramified cohomology group $H^3_{\mathrm{ur}} (K(X), \mu_m^{\otimes 2})$ is finite if and only if $CH^2 (X)/m$ is finite.}

\end{thm}

These results will be proved in \S \ref{S-4}. We note that by work of Kato, Jannsen, Kerz, and Saito, more precise results are available for varieties over finite fields as well as varieties over local fields having good reduction --- see Remark 3.2.  The main tools involved will be some finiteness statements for Galois and \'etale cohomology, together with Bloch-Ogus theory,
which allows one to relate unramified cohomology to \'etale cohomology. As will be apparent, the methods owe much of their inspiration to the work of Colliot-Th\'el\`ene \cite{CT-SB}. One of our motivations was in fact to investigate to what extent the classical finiteness statements over finite and $p$-adic fields can be extended to more general situations, in particular to some base fields of large cohomological dimension.

\vskip5mm

The paper is organized as follows. In \S \ref{S-2}, we discuss fields of type $\tF$ and $\Fpm$ and establish some finiteness results for their Galois and \'etale cohomology. In \S \ref{S-3}, we review the definition and several properties of Kato complexes as well as the main elements of Bloch-Ogus theory needed for the analysis of unramified cohomology.
Finally, in \S \ref{S-4}, we combine these results to prove Theorem \ref{T-MainThm}. We conclude by outlining some connections between finiteness questions for unramified cohomology and Bass's conjecture on the finite generation of $K$-groups.

\vskip10mm

\noindent {\bf Notations and conventions.} Let $X$ be a scheme. For any positive integer
$n$ invertible on $X$, we let $\mu_n = \mu_{n, X}$ be the \'etale sheaf of $n$th roots of unity on $X$.
We follow the usual notations for the Tate twists of $\mu_n$. Namely, for $i \geq 0$, we set
$$
\Z / n \Z (i) = \mu_n^{\otimes i},
$$
(where $\mu_n^{\otimes i}$ is the sheaf associated to the $i$-fold tensor product of $\mu_n$), with the convention that
$$
\mu_n^{\otimes 0} = \Z/n \Z.
$$
If $i < 0$, we let
$$
\Z / n \Z (i) = \Hom (\mu_n^{\otimes (- i)}, \Z / n \Z).
$$

In the case that $X = \mathrm{Spec}~F$ for a field $F$, we identify $\mu_n$ with the group of $n$th roots of unity in a fixed separable closure $\bar{F}$ of $F$. We will also tacitly identify the \'etale cohomology of $X$ with the Galois cohomology of $F$.

Next, for any $j \geq 0$, we denote by $X_j$ and $X^{(j)}$ the sets of points of $X$ of dimension $j$ and codimension $j$, respectively. Given a point $x \in X$, we set $\kappa(x)$ to be the corresponding residue field. Finally, for a variety $X$ over a field, we denote by $CH^2(X)$ the Chow group of cycles of codimension 2 modulo rational equivalence (see, e.g. \cite[Ch. 1]{Fult}).

\vskip5mm

\noindent {\bf Acknowledgements.} I was supported by an NSF Postdoctoral Fellowship at Harvard University during the preparation of this paper. I would like to thank M.~Kerz for helpful communications regarding Kato complexes. Also, the hospitality of IHES in July 2015 is gratefully acknowledged.

\section{Finiteness conditions for Galois and \'etale cohomology}\label{S-2}

In this section, we prove some finiteness results for Galois and \'etale cohomology (generalizing those in \cite[Ch. III, \S 4]{SerreGC}) that will be applied to unramified cohomology in \S \ref{S-4}. For a field $K$, we will denote by $\bar{K}$ a separable closure and by $G_K = \Ga (\bar{K}/K)$ the absolute Galois group.


Let $G$ be a profinite group. Recall that $G$ is said to satisfy Serre's condition $\tF$ if

\vskip2mm

$\mathrm{(F)}$ \ \ \ \parbox[t]{15.5cm}{{\it For every integer $m \geq 1$, $G$ has finitely many open subgroups of index $m$.}}

\vskip2mm
\noindent Furthermore, one says that a field $K$ is {\it of type $\tF$} if $K$ is perfect and $G_K$ satisfies condition $\tF.$ Notice that by the Galois correspondence, this is equivalent to the fact that for every integer $m$, $\bar{K}$ contains finitely many degree $m$ extensions of $K$.


As shown in \cite[Ch. III, \S 4, Proposition 9]{SerreGC}, if $G_K$ is topologically finitely generated, then $K$ is automatically of type $\tF.$ We also note the following observation, which can be used to expand the list of examples of fields of type $\tF$ given in \cite[Ch. III, \S 4.2]{SerreGC}.

\begin{lemma}\label{L-Topfg}
Let $K$ be a field of characteristic 0 such that the absolute Galois group $G_K = \Ga (\bar{K}/K)$ is topologically finitely generated. Then for the field $L = K((t))$ of formal Laurent series over $K$, the absolute Galois group $G_L = \Ga (\bar{L}/ L)$ is also topologically finitely generated.
\end{lemma}
\begin{proof}
Let $F = L^{\text{ur}}$ be the maximal unramified extension of $L$. It is well known that $F$ is a discretely valued henselian field having residue field $\bar{K}$, and that $F/L$ is a Galois extension whose Galois group $\Ga (F/L)$ is naturally identified with $G_K.$ Note that in this case,
$F$ can simply be described as
the compositum of $L$ and $\bar{K}$ inside $\bar{L}$, or, equivalently, the union of the fields $K' ((t))$ for all finite extensions $K'/K.$

To prove the claim, it obviously suffices to show that if $F'/F$ is a field extension of degree $n$, then $F' = F(\sqrt[n]{t}).$ Indeed, this will imply that $H = \Ga (\bar{L}/F ) \simeq \widehat{\Z}$, hence $G_L$ is topologically finitely generated, as needed. Now, a degree $n$ extension $F'/F$ is totally tamely ramified, and consequently is of the form $F(\sqrt[n]{\pi})$ for {\it some} uniformizer $\pi \in F$ (see, e.g., \cite[Ch. II, Proposition 3.5]{FV}). The required statement easily follows from the fact that the residue field of $F$ is algebraically closed.
\end{proof}

Another important property of fields of type $\tF$ is given in the next statement, which is a special case of \cite[Ch. III, \S 4, Proposition 8]{SerreGC}.

\begin{lemma}\label{L-CondF}
Suppose $K$ is a field of type $\tF.$ Then for any integer $m \geq 1$ prime to $\text{char}~K$ and any finite (separable) extension $L/K$, the quotient $L^{\times}/{L^{\times}}^m$ is finite. In other words, $K$ satisfies $\Fpm$ for all $m$ invertible in $K.$
\end{lemma}
\begin{proof}
It is clear from the definitions that if $K$ is of type $\tF$, then so is any finite extension $L/K$. Thus, it suffices to show that for any $m$ prime to $\text{char}~K$, the group
$$
H^1 (K, \Z / m \Z (1)) = K^{\times}/ {K^{\times}}^m
$$
is finite.

Let $L$ be a finite Galois extension of $K$ containing a primitive $m$th root of 1, so that
$\Z / m \Z (1)$ becomes isomorphic (as a $G_L$-module) to $\Z/ m \Z$ with trivial action. It follows that we have the identification
$$
H^1 (L, \Z / m \Z (1)) = \Hom_{\text{cont}} (G_L, \Z / m \Z),
$$
and the latter group is finite since $L$ is of type $\tF.$
The inflation-restriction sequence
$$
0 \to H^1 (\Ga (L/K), \Z / m \Z (1)) \to H^1 (K, \Z / m \Z (1)) \to H^1 (L, \Z / m \Z (1))
$$
then shows that $H^1(K, \Z / n \Z (1))$ is finite as well.

\end{proof}

\vskip3mm




\noindent {\bf Remark 2.3.} It follows from the definition that if $K$ is of type $\Fpm$, then $K$ is automatically of type $\mathrm{(F'_{\ell})}$ for every $\ell \vert m.$


\newpage

\noindent {\bf Example 2.4.}

\noindent (a) \parbox[t]{16cm}{In each of the following cases: $K$ is an algebraically closed field, the field $\R$ of real numbers, a finite field $\mathbb{F}_q$, or a $p$-adic field (i.e. a finite extension $K$ of $\Q_p$), the absolute Galois group $G_K$ is topologically finitely generated, hence $K$ is of type $\tF.$
In fact, in the last case, if $[K:\Q_p] = n$, then $G_K$ can be topologically generated by $n +2$ elements (see \cite[Theorem 3.1]{J1}).}

\vskip1mm

\noindent (b) \parbox[t]{16cm}{If $k$ is a field of characteristic 0 from the list given in (a), then,
using Lemma \ref{L-Topfg} and arguing by induction, we see that
the field $K = k ((t_1)) ((t_2)) \cdots ((t_r))$ of iterated Laurent series over $k$ is of type $\tF.$}

\vskip1mm

\noindent (c) \parbox[t]{16cm}{The field $\mathbb{F}_p ((t))$ is of type $\Fpm$ for all $m$ prime to $p.$ On the other hand, by Artin-Schreier theory, $\mathbb{F}_p ((t))$ has infinitely many cyclic Galois extensions of degree $p$, so is not of type $\tF.$}

\vskip5mm

\noindent {\bf Remark 2.5.} It is well-known that a $p$-adic field $k$ has cohomological dimension 2 (see, e.g., \cite[Ch.~II, \S 4.3, Proposition 12]{SerreGC}); arguing by induction, one then shows that the cohomological dimension of $K = k ((t_1))((t_2)) \cdots ((t_r))$ is $r+2$ (\cite[Theorem 2.5]{Raskind}).

\addtocounter{thm}{3}

\vskip5mm

The main motivation for introducing condition $\Fpm$ in the present context is the following finiteness result for Galois cohomology, which generalizes \cite[Ch. II, \S 5, Proposition 14]{SerreGC}.

\begin{prop}\label{P-FieldF'}
Let $K$ be a field and $m \geq 1$ an integer prime to $\mathrm{char}~K.$ Assume that $K$ is of type $\Fpm.$ Then for any finite $G_K$-module $A$ such that $mA = 0$, the groups $H^i (K, A)$ are finite for all $i \geq 0$.
\end{prop}

One of the principal ingredients in the proof is the Bloch-Kato Conjecture, established in the work of Rost \cite{Rost}, Voevodsky (\cite{V1}, \cite{V2}), Weibel \cite{Weib}, and others. We recall that for a field $k$ and integer $n > 1$, the {\it $n$th Milnor $K$-group} $K_n^M(k)$ is defined as the quotient of the $n$-fold tensor product $k^{\times} \otimes_{\Z} \cdots \otimes_{\Z} k^{\times}$ by the subgroup generated by elements $a_1 \otimes \cdots \otimes a_n$ such that $a_i + a_j = 1$ for some $1 \leq i < j \leq n.$ For $a_1, \dots, a_n \in k^{\times}$, the image of $a_1 \otimes \cdots \otimes a_n$ in $K^M_n(k)$, denoted $\{ a_1, \dots, a_n \}$, is called a {\it symbol}; thus, $K^M_n (k)$ is generated by symbols.
By convention, one sets $K^M_0 (k) = \Z$ and $K^M_1 (k) = k^{\times}.$ Furthermore, for any integer $m$ prime to $\text{char}~k$, the Galois symbol yields a map
$$
h^n_{k,m} \colon K_n^M (k) \to H^n (k, \Z / m \Z (n)).
$$
(See \cite[\S 4.6]{GiSz} for the relevant definitions.) The main result is
\begin{thm}\label{T-BlochKato}
For any field $k$ and any integer $m$ prime to $\text{char}~k$, the Galois symbol induces an isomorphism
$$
K_n^M(k)/m \stackrel{\sim}{\longrightarrow} H^n (k, \Z / m \Z (n))
$$
for all $n \geq 0$.
\end{thm}

\vskip2mm

Let us now return to

\vskip2mm

\noindent {\it Proof of Proposition \ref{P-FieldF'}.} We first show that the groups $H^i (K', \Z / m \Z (1))$ are finite for any finite separable field extension $K'/K$. By Remark 2.3,
this will also establish the finiteness of $H^i (K', \Z / \ell \Z (1))$ for every $\ell \vert m.$
For ease of notation, we will work with $K' = K$, but the argument is identical for any finite extension.

We may assume without loss of generality that $K$ contains a primitive $m$-th root of unity $\zeta_m$. Indeed, if $L = K(\zeta_m)$ and the groups $H^i (L, \Z / m \Z (1))$ are finite, then the finiteness of $H^i (K, \Z / m \Z (1))$ follows from the Hochschild-Serre spectral sequence
$$
H^i (\Ga (L/K), H^j (L, \Z / m \Z (1))) \Rightarrow H^{i+j} (K, \Z / m \Z (1)).
$$
Thus, we may, and do, assume that $\Z / m \Z (1)$ is isomorphic as a $G_K$-module to $\Z / m \Z$ with trivial action.
Then Theorem \ref{T-BlochKato} yields the finiteness of the groups $H^i (K, \Z / m \Z (1))$ for all $i.$ More precisely, by our assumption, $K_1^M (K)/m = K^{\times}/ {K^{\times}}^m$ is finite. Hence, there are only finitely many symbols in $K_i^M(K)/m$, and since each symbol has order at most $m$, it follows that $K_i^M(K)/m$ finite.
Consequently, $H^i (K, \Z / m \Z (1))$ is finite in view of the isomorphisms
$$
K_i^M (K)/m \simeq H^i (K, \Z / m \Z (i)) \simeq H^i (K, \Z / m \Z (1)).
$$

Now let $A$ be an arbitrary finite $G_K$-module annihilated by $m.$ There exists a finite Galois extension $L/K$ such that $A$ becomes isomorphic as a $G_L$-module to a direct sum of modules of type $\Z / \ell \Z (1)$ with $\ell \vert m.$ By the above argument, all of the $H^i (L, A)$ are finite, and again we conclude from the Hochschild-Serre spectral sequence that all of the $H^i (K, A)$ are finite, as needed. $\Box$

\vskip2mm

This result now enables us to derive the following finiteness result for \'etale cohomology.

\begin{cor}\label{C-FinEt}
Let $K$ be a field, $m \geq 1$ an integer invertible in $K$, and assume that $K$ is of type $\Fpm.$
Then for any geometrically integral algebraic variety $X$ over $K$, the groups $\he^i(X, \Z / m \Z (j))$ are finite for all $i$ and $j$.
\end{cor}
\begin{proof}
Let $\bar{X} = X \times_K \bar{K}.$ It is well-known that $\he^i (\bar{X}, \Z / m \Z (j))$ are finite $m$-torsion groups for all $i$ and $j$ (see \cite[Expos\'e XVI, Th\'eor\`eme 5.2]{SGA4}). Consequently, the groups $H^p (K, \he^q (\bar{X}, \Z / m \Z (j))$ are finite by Proposition \ref{P-FieldF'}. Our claim now follows from the
Hochschild-Serre spectral sequence
$$
H^p (K, H^q_{\text{\'et}} (\bar{X}, \Z / m \Z (j))) \Rightarrow H^{p+q}_{\text{\'et}} (X, \Z / m \Z (j)).
$$
\end{proof}

\section{Kato complexes, Bloch-Ogus theory, and unramified cohomology}\label{S-3}

Our goal in this section is to recall some key elements of the theory of Kato complexes and Bloch-Ogus theory that will be needed for our analysis of unramified cohomology.
Throughout this section, we let $F$ be an arbitrary field and fix a positive integer $m$ that is invertible in $F$.

\vskip5mm

\subsection{Kato complexes}(Cf. \cite{Kato}). Let $X$ be an excellent noetherian scheme and $m$ a positive integer invertible on $X$ (note that Kato also deals with the case when $m$ is not necessarily invertible on $X$, but we will not need this). For any integers $i, j$, Kato constructed a homological complex $C_m^{i,j}(X)$
$$
\cdots \to \bigoplus_{x \in X_p} H^{p+ i} (\kappa(x), \Z / m \Z (p + j)) \to \bigoplus_{x \in X_{p-1}} H^{p+i-1} (\kappa(x), \Z/ m \Z (p+j-1)) \to \cdots \to \bigoplus_{x \in X_0} H^{i} (\kappa(x), \Z / m \Z (j))
$$
where the term $\oplus_{x \in X_p}$ is placed in degree $p$.
The differentials
$$
\partial_p \colon \bigoplus_{x \in X_{p}} H^{r+i}(\kappa(x), \Z / m \Z (p + j)) \to \bigoplus_{x \in X_{p-1}} H^{p+i-1}(\kappa(x), \Z / m \Z (p + j-1)).
$$
are defined as follows. Let $x \in X_p$ and set $Z_x = \overline{ \{ x \}}.$ Then each point $y$ of codimension 1 on $Z_x$ corresponds to a point in $X_{p-1}.$ Let $y_1, \dots, y_s$ be the points on the normalization $\tilde{Z}_x$ lying above $y.$ The local ring at each $y_k$ is a discrete valuation ring, yielding a discrete valuation on the function field $\kappa(x)$ of $\tilde{Z}_x.$ Let
$$
\partial^x_{y_k} \colon H^{p+i} (\kappa(x), \Z / m \Z (p+j)) \to H^{p+i-1} (k(y_k), \Z / m \Z (p+j-1))
$$
be the corresponding residue map.
One then defines
$$
\partial^x_y = \sum_{k=1}^s {\rm Cor}_{\kappa(y_k)/ \kappa (y)} \circ \partial^x_{y_k},
$$
where ${\rm Cor}$ is the corestriction map. The differential $\partial_r$ is the direct sum of all such maps. The fact that this construction does indeed produce
a complex is proved in \cite[Proposition 1.7]{Kato}. We will write $H_q (C_m^{i,j}(X))$ for the degree $q$ homology of $C_m^{i,j}(X).$

\vskip5mm

\noindent {\bf Remark 3.1.} If in the above construction, $X$ is a $d$-dimensional smooth irreducible algebraic variety over $F$ with function field $F(X)$, then
from the description of the differentials, we see that the $d$th homology of $C^{i,j}_m(X)$
$$
H_d (C^{i,j}_m(X))= \ker \left( H^{d+i} (F(X), \Z / m \Z (d+j)) \to \bigoplus_{x \in X_{d-1}} H^{d+i-1} (\kappa(x), \Z / m \Z (d+j-1)) \right)
$$
is precisely $H_{\text{ur}}^{d+i} (F(X), \Z / m \Z (d+j)).$

\vskip5mm

\noindent {\bf Remark 3.2.} We mention some important results regarding Kato complexes of varieties over finite and local fields.

\vskip2mm

\noindent (a) \parbox[t]{16cm}{Let $X$ be a proper smooth irreducible algebraic variety over a finite field $L.$ Then $H_q(C_m^{1,0}(X)) = 0$ for all $q \neq 0$ and any $m$ invertible in $L$. Indeed, for $X$ of dimension 1, this is simply the function field version of the Brauer-Hasse-Noether theorem. The vanishing of $H_2(C_m^{1,0}(X))$ for 2-dimensional varieties was established by Colliot-Th\'el\`ene, Soul\'e, and Sansuc (see \cite[\S 2.4, Rem. 2]{CTSS}). The general case was conjectured by Kato (see \cite[Conjecture 0.3]{Kato}) and proved by Kerz and Saito \cite{KS}, building on earlier work of Jannsen and Saito \cite{JS}.}

\vskip2mm

\noindent (b) \parbox[t]{16cm}{Let $K$ be a local field with ring of integers $\mathcal{O}_K$ and (finite) residue field $k.$ Suppose $X$ is a regular proper flat scheme over $\mathcal{O}_K$. Denote by $X_{\eta}$ and $X_s$ the generic and special fibers of $X$, respectively, and assume that $X_{\eta}$ is smooth. Since $X_p \cap X_{\eta} = (X_{\eta})_{p-1}$, it follows that the differential in the Kato complex of $X$ induces a morphism of complexes
\begin{equation}\label{E-KatoComplexResidue}
\partial_X \colon C_m^{2,1}(X_{\eta}) \to C_m^{1,0}(X_s)^{(-)},
\end{equation}
where the superscript $(-)$ indicates that the differentials of the complex have been multiplied by -1. Kato conjectured that $(\ref{E-KatoComplexResidue})$ is a quasi-isomorphism of complexes and verified this to be the case when $\dim X = 2$ (\cite[Proposition 5.2]{Kato}). The general case was handled by Kerz and Saito (see \cite[proof of Theorem 8.4]{KS}).
Combined with (a), this implies that if $S$ is a smooth proper irreducible algebraic variety over $K$ having good reduction, then $H_q (C_m^{2,1}(S)) = 0$ for $q \neq 0$ and any positive integer $m$ invertible in $k$.}






\addtocounter{thm}{2}

\vskip5mm

\subsection{Bloch-Ogus theory} We now discuss a more geometric approach, due to Bloch and Ogus \cite{BO}, to studying unramified cohomology.
In algebraic topology, it is well-known that the homology of a CW-complex $Z$ can be analyzed by considering the filtration by $i$-skeleta $Z_i$:
$$
\emptyset = Z_{-1} \subset Z_0 \subset Z_1 \subset \cdots \subset Z_k = Z.
$$
More precisely, such a filtration induces a natural filtration on the singular chains of $Z$, which, using the theory of exact couples, leads to a spectral sequence converging to the homology of $Z$.

As shown in \cite{BO}, a similar strategy can be implemented for algebraic varieties using \'etale homology. Let $S = {\rm Spec}(F)$, for a field $F$, and consider an algebraic variety $f \colon X \to S$ over $F$. Recall that for an integer $m$ invertible in $F$, one sets
$$
\hhe_a (X/F, \Z / m \Z (b)) = \he^{-a}(X, Rf^! \Z / m \Z (-b)),
$$
where $Rf^!$ is the exceptional inverse image functor defined in \cite[Expos\'e XVIII]{SGA4}. A key observations is that, just as in the topological setting,
there is a spectral sequence of homological type (the so-called {\it niveau spectral sequence})
\begin{equation}\label{E-NiveauSS}
E^1_{p,q} (X/F, \Z/ m \Z (b)) = \bigoplus_{x \in X_p} \hhe_{p+q}(x/F, \Z / m \Z (b)) \Rightarrow \hhe_{p+q}(X/F, \Z / m \Z (b)),
\end{equation}
where
$$
\hhe_a (x/F, \Z / m \Z (b)) = \lim_{\rightarrow} \hhe_a (V/F, \Z / m \Z (b))
$$
and the limit is taken over all open non-empty subschemes $V \subset \overline{ \{x  \}}$ (see \cite[Proposition 3.7]{BO}).
Furthermore, if $X$ is a {\it smooth} algebraic variety, then one can pass from \'etale homology to \'etale cohomology using Poincar\'e duality for smooth algebraic varieties (see \cite[Expos\'e XVIII, 3.2.5]{SGA4}). Indeed, given a smooth irreducible variety $X$ of dimension $d$, there are canonical isomorphisms
\begin{equation}\label{E-PDuality}
\hhe_a(X/F, \Z / m \Z (b)) \simeq \he^{2d -a}(X, \Z / m \Z (d-b)).
\end{equation}
In this case, it is convenient to take $b' = d-b$, $q' =d-q$, and $r' = d-r$ and renumber (\ref{E-NiveauSS}) into a cohomological first quadrant {\it coniveau spectral sequence}
\begin{equation}\label{E-ConiveauSS}
E_1^{p,q}(X/F, \Z / m \Z (b)) = \bigoplus_{x \in X^{(p)}} H^{q-p}(\kappa(x), \Z / m \Z (b - p)) \Rightarrow \he^{p+q}(X, \Z / m \Z (b)),
\end{equation}
where the groups on the left are the Galois cohomology groups of the residue fields $\kappa(x)$ (for a slightly different derivation of this spectral sequence that avoids the use of \'etale homology, one can consult \cite{CTHK}).

\vskip5mm

\noindent {\bf Remark 3.3.}
It is well-known that the Bloch-Ogus and Kato constructions are compatible (see, e.g., \cite[Remark 2.5.5]{JSS}). More precisely, if $X$ is a smooth irreducible algebraic variety over a field $F$, then the Bloch-Ogus complex $$E_1^{\bullet, q}(X/F, \Z / m \Z (b))$$ coincides with the sign-modified Kato complex $$C_m^{q-d, b-d}(X)^{(-)}.$$

\vskip5mm

\addtocounter{thm}{1}

The fundamental result of Bloch and Ogus was the calculation of the $E_2$-term of (\ref{E-ConiveauSS}). We will need the following notation: let
$\mathcal{H}^q (\Z / m \Z (j))$
denote the Zariski sheaf on $X$ associated to the presheaf that assigns to an open $U \subset X$ the cohomology group $H^i_{\text{\'et}}(U, \Z / m \Z (j)).$
Bloch and Ogus showed that
$$
E_2^{p,q}(X/F, \Z/ m \Z (b)) = H^p (X, \mathcal{H}^q (\Z / m \Z (b)))
$$
(see \cite[Corollary 6.3]{BO}). The resulting (first quadrant) spectral sequence
\begin{equation}\label{E-BO}
E_2^{p,q}(X/F, \Z / m \Z (b)) = E_2^{p,q}= H^p (X, \mathcal{H}^q (\Z / m \Z (b))) \Rightarrow H^{p+q}_{\text{\'et}} (X, \Z/ m \Z (b))
\end{equation}
is usually referred to as the {\it Bloch-Ogus spectral sequence}.

\vskip2mm

For ease of reference, we summarize the main points of the above discussion.

\begin{prop}\label{P-BlochOgusSS}
The Bloch-Ogus spectral sequence $(\ref{E-BO})$ associated to a smooth irreducible algebraic variety $X$ over a field $F$ has the following properties:
\vskip1mm

\noindent {\rm (a)} $E_2^{p,q} = 0$ for $p > \dim X$ and all $q$;

\vskip1mm

\noindent {\rm (b)} $E_2^{p,q} = 0$ for $p > q$; and

\vskip1mm

\noindent {\rm (c)} \parbox[t]{16cm}{$E_2^{0, q} = H^0 (X, \mathcal{H}^q (\Z / m \Z (b)))$ coincides with the unramified cohomology $H^q_{\mathrm{ur}} (F(X), \Z / m \Z (b))$.}

\end{prop}

\section{Finiteness results for unramified cohomology}\label{S-4}

In this section, we apply the finiteness results of \S \ref{S-2} to the Kato and Bloch-Ogus constructions discussed in \S \ref{S-3} to prove Theorem \ref{T-MainThm}.


We begin with the case of $X = C$ a smooth, geometrically integral curve over a field $F$. Notice that by Proposition \ref{P-BlochOgusSS}(a), we have $E_2^{p,q} = 0$ for $p \neq 0,1$ and all $q.$ Consequently, for each $i \geq 1,$ there is a short exact sequence
$$
0 \to E_2^{1, i-1} \to E^i \to E_2^{0,i} \to 0
$$
(see, e.g., \cite[Ch. XV, \S 5, Proposition 5.5]{CartEil}). Applying Proposition \ref{P-BlochOgusSS}(c), we thus obtain a surjection
\begin{equation}\label{E-EtUnram}
H^i_{\text{\'et}}(C, \Z / m \Z (j)) \twoheadrightarrow H^i_{\text{ur}} (F(C), \Z / m \Z (j))
\end{equation}
for all $i \geq 1$ and all $m$ invertible in $F$.

\vskip2mm

Now, if $F = K$ is a field of type $\Fpm$ (see \S \ref{S-2}), then (\ref{E-EtUnram}) and Corollary \ref{C-FinEt} yield

\begin{prop}\label{P-4.1}
Let $K$ be a field of type $\Fpm$ for an integer $m \geq 1$ invertible in $K$, and $C$ be a smooth, geometrically integral curve over $K$. Then the
unramified cohomology groups $H^i_{\mathrm{ur}} (K(C), \Z/ m \Z (j))$ are finite for all $i \geq 1.$
\end{prop}

For smooth algebraic varieties $X$ of dimension $\geq 2$, the edge map in (\ref{E-EtUnram}) is in general no longer surjective.
However, in this case, the finiteness of unramified cohomology in degree 3 is closely related to the question of finite generation of the Chow group $CH^2 (X)$ of codimension 2 cycles modulo rational equivalence.

First, we note the following general fact about first quadrant spectral sequences satisfying condition (b) of Proposition \ref{P-BlochOgusSS}.

\begin{lemma}\label{L-SpecSeqLemma}
Let $E_2^{p,q} \Rightarrow E^{p+q}$ be a first quadrant spectral sequence. Assume that $E_2^{p,q} = 0$ for $p > q$. Then there is an exact sequence
$$
E^3 \stackrel{e}{\longrightarrow} E_2^{0,3} \stackrel{d_2}{\longrightarrow} E_2^{2,2} \stackrel{f}{\longrightarrow} E^4.
$$

\end{lemma}

\noindent (In the statement, we assume, as usual, that for each $n$, the filtration $ \{ F^p E^n \}$ on $E^n$ is such that $F^p E^n = 0$ for sufficiently large $p$ and $F^p E^n = E^n$ for sufficiently small $p$; consequently, since we are working with a first quadrant spectral sequence, we have $F^0 E^n = F^{-1} E^n = F^{-2}E^n = \cdots = E^n$ and $F^{n+1} E^n = F^{n+2} E^n = F^{n+3} E^n = \cdots = 0.$)

\vskip2mm

\noindent

We give the definitions of the maps for completeness; the proof of exactness then easily follows. We take
$e \colon E^3 \to E_2^{0,3}$ to be the usual edge map defined as the composition of the surjection $E^3 \twoheadrightarrow E^{0,3}_{\infty}$ followed by the inclusion $E^{0,3}_{\infty} \hookrightarrow E_2^{0,3}.$ The map $d_2$ is simply the differential of the spectral sequence. The map $f \colon E_2^{2,2} \to E^4$ is obtained as follows. Our assumption implies that
$$
F^3 E^4/ F^4 E^4 = E^{3,1}_{\infty} = 0 \ \ \ \text{and} \ \ \ F^4E^4/ F^5E^4 = E^{4,0}_{\infty} = 0.
$$
Consequently, $F^3 E^4 = F^4 E^4 = F^5 E^4 = \cdots = 0.$ Therefore, we have an inclusion $$E^{2,2}_{\infty} = F^2 E^4/F^3 E^4 = F^2 E^4 \hookrightarrow E^4.$$ Next, it is easy to see that $E_{\infty}^{2,2} = E_3^{2,2}$, and since $E_2^{4,1} = 0$, so that the differential $d_2 \colon E_2^{2,2} \to E_2^{4,1}$ is zero, we obtain a surjection $E_2^{2,2} \twoheadrightarrow E_{\infty}^{2,2}.$ The map $f$ is then defined as the composition
$$
E_2^{2,2} \twoheadrightarrow E_{\infty}^{2,2} \hookrightarrow E^4.
$$

\vskip2mm

Applying Lemma \ref{L-SpecSeqLemma} to the Bloch-Ogus spectral sequence (with $b =2$), we thus obtain, for any smooth algebraic variety $X$ over an arbitrary field $F$ and positive integer $m$ invertible in $F$, an exact sequence
\begin{equation}\label{E-BOChow1}
\he^3 (X, \Z / m \Z (2)) \to H^0 (X, \mathcal{H}^3(\Z / m \Z (2))) \to H^2 (X, \mathcal{H}^2 (\Z / m \Z (2))) \to \he^4 (X, \Z m \Z (2)).
\end{equation}
On the other hand, a well-known consequence of Quillen's proof of Gersten's conjecture in algebraic $K$-theory is the existence of an isomorphism
$$
H^2 (X, \mathcal{H}^2 (\Z / m \Z (2))) \simeq CH^2 (X)/m
$$
(see, e.g., the proof of \cite[Theorem 7.7]{BO}). Thus, using Proposition \ref{P-BlochOgusSS}(c), we may rewrite (\ref{E-BOChow1}) as
\begin{equation}\label{E-BOCHow2}
\he^3 (X, \Z / m \Z (2)) \to H^3_{\text{ur}} (F(X), \Z / m \Z(2)) \to CH^2 (X)/m \to \he^4 (X, \Z m \Z (2)).
\end{equation}
In view of Corollary \ref{C-FinEt}, we therefore have

\begin{prop}\label{P-4.3}
Let $K$ be a field of type $\Fpm$ for an integer $m \geq 1$ invertible in $K$, and $X$ a smooth, geometrically integral algebraic variety over $K$. Then the
unramified cohomology group $H^3_{\mathrm{ur}} (K(X), \Z/ m \Z (2))$ is finite if and only if $CH^2 (X)/m$ is finite. In particular, $H^3_{\mathrm{ur}} (K(X), \Z/ m \Z (2))$ is finite if $CH^2(X)$ is finitely generated.
\end{prop}

\vskip2mm

\noindent The two statements of Theorem \ref{T-MainThm} are now contained in Propositions \ref{P-4.1} and \ref{P-4.3}, which concludes the proof.

\vskip5mm

\noindent {\bf Remark 4.4.} The question of the finite generation of $CH^2(X)$ (or, at least, the finiteness of $CH^2(X)/m$) is in general a wide-open problem; we mention a couple of cases where the affirmative answer is known.

\vskip2mm

\noindent (a) \parbox[t]{16cm}{(cf. \cite[\S 4.3]{CT-SB}) Let $K$ be a field of characteristic 0 of type $\Fpm$ for $m \geq 1$ and $X$ a smooth geometrically, integral algebraic variety of dimension $d$ over $K$. Write $\bar{X} = X \times_K \bar{K}$ and suppose there exists a dominant rational map
$$
\mathbb{A}_{\bar{K}}^{d-1} \times_{\bar{K}} C \dashrightarrow \bar{X},
$$
where $C/ \bar{K}$ is an integral curve. Then, using Corollary \ref{C-FinEt}, the argument in \cite[Theorem 4.3.7]{CT-SB} shows that $CH^2(X)/m$ is finite. One of the key points is that for any smooth algebraic variety $U$ over $K$, the group ${}_mCH^2(U)$ is finite: indeed, as observed in \cite[Corollaire 2]{CTSS}, the Merkurjev-Suslin theorem implies that ${}_mCH^2(U)$ is a subquotient of $\he^3 (U, \Z / m \Z (2)).$}

\vskip2mm

\noindent (b) \parbox[t]{16cm}{Let $X$ be a noetherian scheme. Recall that $CH_0(X)$ is defined as the cokernel of the natural map
$$
\bigoplus_{x \in X_1} K_1 (\kappa(x)) \to \bigoplus_{x \in X_0} K_0 (\kappa(x))
$$
induced by valuations (or equivalently, coming from the localization sequence in algebraic $K$-theory). One of the key results of higher-dimensional global class field theory is that if $X$ is a regular scheme of finite type over $\Z$, then $CH_0(X)$ is a finitely generated abelian group. This was initially established in the work of Bloch \cite{BlochCFT}, Kato and Saito \cite{Kato-Saito}, and Colliot-Th\'el\`ene, Soul\'e, and Sansuc \cite{CTSS}; a more recent treatment was given by Kerz and Schmidt \cite{KerzSchmidt} based on ideas of Wiesend. In particular, if $X$ is a smooth irreducible algebraic surface over a finite field, then $CH_0(X)$ is finitely generated.}

\vskip5mm

One thing to note is that the last example indicates a connection between finiteness properties of unramified cohomology and Bass's conjecture on the finite generation of algebraic $K$-groups. We conclude by briefly outlining the set-up and refer the reader to \cite{Geiss} for a detailed exposition. Let $X$ be a regular scheme, and denote by $K_i (X)$ the $i$-th Quillen $K$-group of $X$ (defined using either the category of vector bundles or the category of coherent sheaves on $X$). Motivated by Dirichlet's theorems on the finite generation of the unit group and the finiteness of the class group for number fields, Bass \cite{Bass} formulated

\vskip2mm

\noindent {\bf Conjecture 4.5.} {\it Let $X$ be a regular scheme of finite type over $\Z$. Then the groups $K_i (X)$ are finitely generated for all $i \geq 0.$}

\vskip2mm

This conjecture also has a motivic analogue. First, recall that
if $X$ is a smooth algebraic variety over a perfect field $F$, there is a fourth quadrant spectral sequence
$$
E_2^{p,q} = CH^{-q} (X, -p-q) \Rightarrow K_{-p-q}(X)
$$
(see \cite{BL} for the case where $X$ is the spectrum of a field, and \cite{FS} for the general case). The groups $CH^n (X,j)$ are Bloch's higher Chow groups \cite{Bloch}; it follows from their construction that $CH^n (X, 0) = CH^n(X)$ (the usual Chow group of codimension $n$ cycles modulo rational equivalence). Furthermore, as shown by Voevodsky \cite{V}, the higher Chow groups agree with the motivic cohomology groups introduced in his work on the Milnor conjecture. More precisely, if $X$ is a smooth algebraic variety over a perfect field, then
$$
H^i_{\mathcal{M}} (X, \Z (n)) \simeq CH^n (X, 2n - i).
$$
In view of these results, it is natural to formulate the following refinement of Bass's conjecture:

\vskip2mm

\noindent {\bf Conjecture 4.6.} {\it Let $X$ be a regular scheme of finite type over $\Z$. Then the groups $H^i_{\mathcal{M}} (X, \Z (n))$ are finitely generated.}

\vskip2mm

\noindent Though some progress has been made on these conjectures, definitive results are fairly limited. Of particular note is Quillen's analysis of the $K$-groups of rings of integers of number fields \cite{Quil1} and of finite fields \cite{Quil2}, which setted Conjecture 4.5 in these two cases.


\addtocounter{thm}{3}

Regarding Conjecture 4.6, as a corollary of their proof of Kato's conjectures for smooth proper varieties over finite fields, Jannsen, Kerz, and Saito have obtained the following result:

\begin{thm}{\rm (cf. \cite{JS}, \cite{KS})}
Let $X$ be a smooth proper algebraic variety of dimension $d$ over a finite field $F$. Then the groups $CH^d(X,j)/m$ are finite for all integers $m$ invertible in $F$ and all $j.$
\end{thm}

\bibliographystyle{amsplain}

\end{document}